\documentclass[11pt]{amsart}
\usepackage{preamble}
\usepackage{tikz-cd}

\newcommand{\Aut}{\operatorname{Aut}}

\newcommand{\gr}{\operatorname{gr}}

\begin{document}

\title[On singular fibers of parabolic fibrations]{On singular fibers of parabolic fibrations}


\begin{abstract}
We describe the singular fibers of a parabolic fibration $f:X\to Y$ whose moduli divisor $M_Y$ is numerically trivial and whose discriminant divisor $B_Y$ is zero.
\end{abstract}

\author{Yiming Zhu}
\address{School of Mathematical Science, University of Science and Technology of China, Hefei 230026, P.R.China.} \email{zhuyiming@ustc.edu.cn}

\maketitle

\setcounter{tocdepth}{1}
\tableofcontents

\section{Introduction}
Let $f:S\to C$ be an elliptic fibration over a curve. By successively contracting $(-1)$-curves on fibers, we get a birational morphism $\pi:S\to S'$ to a smooth surface $S'$, so that the resulting elliptic fibration $f':S'\to C$ has no $(-1)$-curves on its fibers. Kodaira's canonical bundle formula of $f'$ writes:
\begin{align*}
K_{S'}=f'^*(K_C+L+\sum_P(1-{1}/{m_P})P),    
\end{align*}
where the sum runs over all points $P$ such that the fiber $f'^*P$ is a multiple fiber of multiplicity $m_P$, and $L$ is a divisor associated to the line bundle $f'_*O_{S'}(K_{S'/C})$. Assume that $\deg L=0$ and $\sum_P(1-{1}/{m_P})P=0$. Then $f'$ is locally trivial by \cite[Chapter 3, Theorem 18.2]{BHPV2004}. In particular, every fiber $f'^*P$ is reduced and irreducible. Now, consider the original fibration $f:S\to C$. Write $ K_S=\pi^*K_{S'}+B$. Then $B$ is an effective divisor that does not dominate $C$. Let $f^*P=\sum_ia_iQ_i$ be a fiber of $f$, where the $Q_i$ are prime divisors, and let $f^*P\setminus B:=\sum_{i: Q_i\not\subset B}a_iQ_i$. Then $f^*P\setminus B$ is reduced and irreducible.

In this note, we consider the higher-dimensional case. Following Ueno, we call a fibration parabolic if its general fiber has Kodaira dimension zero. For parabolic fibrations, we have the following canonical bundle formula of Fujino and Mori.

\begin{thm}[\cite{FujinoMori2000}]\label{thm:fujino-mori}
Let $f:X\to Y$ be a parabolic fibration between smooth projective varieties and let $F$ be its general fiber. Let $b=\min\{k\in\Zz_{>0}\mid \dim H^0(kK_F)>0\}$. Then there are $\Qq$-divisors $B_Y$ (the discriminant divisor), $M_Y$ (the moduli divisor), and $B_X$ such that
\begin{align*}
bK_X=f^*b(K_Y+M_Y+B_Y)+B_X,    
\end{align*}
with the following properties:
\begin{enumerate}
    \item If we write $B_X=B_X^{+}-B_X^{-}$ as a difference of two effective $\Qq$-divisors without common components, then $f_*O_X(\lf mB_X^{+}\rf)=O_Y$ for all positive integers $m$, and $\codim f(B_X^{-})\geq2$.
    \item $B_Y=\sum_P(1-t_P)P$, where $t_P=\lct(X,-\frac{1}{b}B_X,f^*P)$, the log canonical threshold of $f^*P$ with respect to the pair $(X,-\frac{1}{b}B_X)$ near the generic point of $P$.
\end{enumerate}

\end{thm}

The following is the first main result of this note.

\begin{thm}\label{thm:MMP to isotrivial}
Let $f:X\to Y$ be a parabolic fibration between smooth projective varieties. Assume that $M_Y$ is numerically trivial and $B_Y=0$. If $X$ has a good minimal model $X'$ over $Y$, and $f':X'\to Y$ is the induced fibration with general fiber $F'$, then there is a finite \'etale cover $Y'\to Y$ such that $X'\times_YY'$ is isomorphic to $Y'\times F'$ over $Y'$.
\end{thm}

Without the good minimal model assumption, we prove the following.

\begin{thm}\label{thm:sing fiber}
Let $f:X\to Y$ be a parabolic fibration between smooth projective varieties. Assume that $M_Y$ is numerically trivial and $B_Y=0$. Then 
\begin{enumerate}
\item $B_X$ is effective,
\item any prime divisor $Q$ on $X$ with $\codim f(Q)\geq 2$ is a component of $B_X$, and 
\item for any prime divisor $P$ on $Y$, the divisor $f^*P\setminus B_X$ is reduced and irreducible.  
\end{enumerate}
\end{thm}

The proof uses the following theorem, which is a consequence of the Clemens--Schmid sequence.

\begin{thm}\label{thm:sum of pg}
Let $f:X\to C$ be a semistable fibration from a smooth projective variety $X$ to a smooth projective curve $C$. 

Let $\sum_iX_i$ be a singular fiber and let $X_{t_0}=f^{-1}(t_0)$ be a general fiber. If $\deg f_*O_X(K_{X/C})=0$, then $\sum_i \dim H^0(K_{X_i})=\dim H^0(K_{X_{t_0}})$.
\end{thm}

Parabolic fibrations with $M_Y$ numerically trivial and $B_Y=0$ are closely related to Ueno's conjecture.

\begin{conj}[{Ueno's conjecture K}]\label{conjK}
Let $X$ be a smooth projective variety of Kodaira dimension zero. Let $f:X\to Y$ be the Albanese morphism and let $F$ be a general fiber of $f$. Then 
 \begin{enumerate}
     \item $f$ is a fibration,
     \item $\kappa(F)=0$, and 
     \item there is a finite \'etale cover $Y'\to Y$ such that $X\times_YY'$ is birational to $Y'\times F$ over $Y'$.
 \end{enumerate}
\end{conj}Part \textit{(1)} was proved by Kawamata \cite[Theorem 1]{Kawamata1981Characterization}. Part \textit{(2)} follows as a corollary of Cao-Păun's solution to the Iitaka conjecture in the case where the base is an abelian variety \cite[Theorem 1.1]{CapPaun2017KodairaDimension}. Let $f$ be the parabolic fibration appearing in Conjecture \ref{conjK}, and let $M_Y$ and $B_Y$ denote the moduli divisor and the discriminant divisor of $f$, respectively. By \cite[Theorem 5.2]{HaconPopaSchnell2018}, we have that $M_Y$ is numerically trivial and $B_Y=0$ (cf. \cite[Lemma 2.1]{zhu2024effective}). Consequently, Theorems \ref{thm:MMP to isotrivial} and \ref{thm:sing fiber} apply to this fibration.

\subsection*{Notation and conventions} 
We work over the field $\Cc$ of complex numbers. A surjective projective morphism between normal varieties with connected fibers is called a fibration. A fibration is called parabolic if its general fiber has Kodaira dimension zero.

Let $f:X\to Y$ be a fibration between normal varieties. Let $D$ be an effective $\Qq$-divisor. Then we can write $D=D^h+D^v$, where $D^v$ does not dominate $Y$, and every component of $D^h$ dominates $Y$. We call $D$ a vertical divisor (resp. horizontal divisor) if $D=D^v$ (resp. $D=D^h$). If $D$ is vertical, it is called exceptional over $Y$ if $\codim f(D)\geq2$, and very exceptional over $Y$ if for any prime divisor $P\subset Y$, there exists a prime divisor $Q\subset X$ such that $Q$ dominates $P$ but $Q\not\subset D$.

\section{Proof of Theorem \ref{thm:MMP to isotrivial}}

\begin{lem}[{\cite[Lemma 2.15]{Gongyo2013Reduction}}]\label{lem:criterion of very exceptional}
Let $f:X\to Y$ be a fibration between normal varieties. Assume that $Y$ is smooth. Let $D$ be a vertical effective $\Qq$-divisor. If $(f_*O_X(\lf mD \rf))^{\wedge}=O_Y$ (where $\wedge$ denotes the reflexive hull) for all sufficiently divisible positive integers $m$, then $D$ is very exceptional over $Y$.
\end{lem}
\begin{proof}
If $D$ is not very exceptional over $Y$, then there is a prime divisor $P\subset Y$, an effective divisor $E$ that is exceptional over $Y$, and a positive integer $N$, such that $f^*P\leq N(E+D)$. We may assume that $N(E+D)$ is a $\Zz$-divisor. Then $O_Y(mP)\subset f_*O_X(mN(E+D))\subset(f_*O_X(mND))^{\wedge}=O_Y$ for all sufficiently divisible positive integers $m$, which is a contradiction.
\end{proof}

\begin{lem}\label{lem:very exceptional}
Let $f:X\to Y$ be a fibration between normal varieties with a general fiber $F$. Assume that $Y$ is smooth. Let $D_1$ be an effective horizontal $\Zz$-divisor on $X$ such that $\dim H^0(D_1|_F)=1$, let $D_2$ be an effective very exceptional $\Qq$-divisor on $X$, and let $L$ be a $\Qq$-divisor on $Y$. Then $f_*O_X(\lf f^*L+D_1+D_2\rf)=O_Y(\lf L\rf)$.

\end{lem}
\begin{proof}
Let $U\subset Y$ be an affine open subset, and let $\phi\in H^0(f^{-1}U,f^*L+D_1+D_2)\subset\Cc(X)$ be a section. Then \begin{align*}
\di_X(\phi)+f^*L+D_1+D_2|_{f^{-1}U}\geq0.    
\end{align*} By restricting to $F$, we have 
\begin{align*}
\di_X(\phi)+D_1|_F\geq0.    
\end{align*} Since $\dim H^0(D_1|_F)=1$, $\phi|_F$ is a constant, which means that $\phi=f^*\psi$ for some $\psi\in\Cc(Y)$. Then by \begin{align*}
f^*\di_Y(\psi)+f^*L+D_1+D_2|_{f^{-1}U}\geq0,    
\end{align*} we have \begin{align*}
\di_Y(\psi)+L|_U\geq0,    
\end{align*} so $\psi\in H^0(U,\lf L\rf)$. Hence $f_*O_X(\lf f^*L+D_1+D_2\rf)\subset O_Y(\lf L\rf)$. On the other hand, we have \[O_Y(\lf L\rf)=f_*O_X(f^*\lf L\rf)\subset f_*O_X(\lf f^* L\rf) \subset f_*O_X(\lf f^* L+D_1+D_2\rf).\] Hence the Lemma follows.
\end{proof}

\begin{lem}\label{lem:Eff of BX}
Let $f:X\to Y$ be a parabolic fibration between smooth projective varieties. If $M_Y\equiv0$ and $B_Y=0$, then 
\begin{enumerate}
    \item $bM_Y$ is a $\Zz$-divisor, 
    \item $O_Y(mbM_Y)=f_*O_X(mbK_{X/Y})$ for all $m\in\Zz_{>0}$, and 
    \item $B_X$ is an effective $\Zz$-divisor.
\end{enumerate}
\end{lem}
\begin{proof}

Consider the non-zero morphism
$f^*f_*O_X(bK_{X/Y})\to O_X(bK_{X/Y})$. We claim that $f_*O_X(bK_{X/Y})$ is locally free, and if $E$ is the effective $\Zz$-divisor such that \begin{align}\label{equ:E}O_X(bK_{X/Y})=f^*f_*O_X(bK_{X/Y})\otimes O_X(E)
\end{align}
holds, then the vertical part of $E$ is very exceptional over $Y$. 

Assume that this claim holds. Then by Lemma \ref{lem:very exceptional}, we have \[f_*O_X(mbK_{X/Y})=(f_*O_X(bK_{X/Y}))^{\otimes m}\otimes f_*O_X(mE)=(f_*O_X(bK_{X/Y}))^{\otimes m}\text{ for all }m\in\Zz_{>0}.\] By the construction of $M_Y$ and $B_Y$ (see \cite[Proposition 2.2]{FujinoMori2000}), we have $O_Y(bM_Y)=f_*O(bK_{X/Y})$, $B_X$ is just the effective $\Zz$-divisor $E$, and the canonical bundle formula is just a rewrite of \ref{equ:E}. 

We now prove the claim. By the construction of $M_Y$ and $B_Y$, there is a positive integer $c$, such that \begin{align*}
((f_*O_X(cbK_{X/Y}))^{\wedge})^{\otimes m}&= (f_*O_X(mcbK_{X/Y}))^{\wedge} \text{ for all }m\in\Zz_{>0}, 
\end{align*}
 and $
O_Y(cb(M_Y+B_Y))=(f_*O_X(cbK_{X/Y}))^{\wedge}$. Since $M_Y\equiv0$ and $B_Y=0$,  we have $(f_*O_X(mcbK_{X/Y}))^{\wedge}\equiv0$ for all $m\in\Zz_{>0}$. Consider the non-zero morphism
\begin{align*}
((f_*O_X(bK_{X/Y}))^{\wedge})^{\otimes mc}\to (f_*O_X(mcbK_{X/Y}))^{\wedge}.    
\end{align*} Since $(f_*O_X(mcbK_{X/Y}))^{\wedge}\equiv0$ and $(f_*O_X(bK_{X/Y}))^{\wedge}$ is pseudo-effective by \cite[Theorem 3]{Viehweg1983WeakPositivity}, we have \begin{align*}
((f_*O_X(bK_{X/Y}))^{\wedge})^{\otimes mc}=f_*O_X(mcbK_{X/Y})^{\wedge}\equiv0.    
\end{align*} By \cite[Corollary 4.6]{HaconPopaSchnell2018}, the torsion-free sheaves $f_*O_X(bK_{X/Y})$, $f_*O_X(mcbK_{X/Y})$ are in fact locally free. Hence \begin{align}\label{equ:1}
f_*O_X(mcbK_{X/Y})=(f_*O_X(bK_{X/Y}))^{\otimes mc}\text{ for all }m\in\Zz_{>0}.
\end{align}
By \ref{equ:1} and \ref{equ:E}, we have  
\begin{align*}
f_*O_X(mcbK_{X/Y})&=(f_*O_X(bK_{X/Y}))^{\otimes mc}\otimes f_*O_X(mcE)=f_*O_X(mcbK_{X/Y})\otimes f_*O_X(mcE)     
\end{align*}for all $m\in\Zz_{>0}$. Hence $f_*O_X(mcE)=O_Y$ for all $m\in\Zz_{>0}$. By Lemma \ref{lem:criterion of very exceptional}, the vertical part $E^v$ of $E$ is very exceptional. 
\end{proof}

\begin{lem}\label{lem:mmp for fibration}
Let $f:X\to Y$ be a fibration between smooth projective varieties. Let $\pi:X\dashrightarrow X'$ be a good minimal model of $X$ over $Y$ and let $f':X'\to Y$ be the induced fibration. Then for general $y\in Y$, the fiber $X'_y$ has terminal singularities, the rational map $X_y\dashrightarrow X'_y$ is $K_{X_y}$-non-positive (see \cite[Definition 3.6.1]{BCHM}), and $K_{X'_y}$ is semi-ample.
\end{lem}
\begin{proof}
Since $X'$ has terminal singularities, the general fiber $X'_y$ has terminal singularities. Let $X\xleftarrow{p}W\xrightarrow{q}X'$ be a common resolution. Write 
\begin{align*}
p^*K_{X}=q^*K_{X'}+E,    
\end{align*} where $E$ is an effective exceptional over $X'$ divisor. By restricting to the fiber $X'_y$, we have 
\begin{align*}
p^*K_{X_y}=q^*K_{X'_y}+E|_{W_y}.    
\end{align*} Since $y$ is general, $E|_{W_y}$ is effective. Hence the rational map $X_y \dashrightarrow X'_y$ is $K_{X_y}$-non-positive.

By the relative semi-ampleness of $K_{X'}$, there exists a positive integer $m$ such that the morphism $f'^*f'_*O_{X'}(mK_{X'})\to O_{X'}(mK_{X'})$ is surjective. By restricting to the fiber $X'_{y}$, we have the following commutative diagram
\begin{center}
\begin{tikzcd}
f'^*f'_*O_{X'}(mK_{X'}) \arrow[d] \arrow[r]   & O_{X'}(mK_{X'}) \arrow[d] \\
f'^*f'_*O_{X'}(mK_{X'})|_{X'_y}\arrow[r] & O_{X'}(mK_{X'})|_{X'_y}  
\end{tikzcd}
\end{center}
For a Zariski open subset $U\subset Y$, the map $y\in U\mapsto \dim H^0(mK_{X'_y})=\dim H^0(mK_{X_y})$ is constant, hence $f'_*O_{X'}(mK_{X'})\otimes\Cc(y)=H^0(mK_{X'_y})$ and $f'^*f'_*O_{X'}(mK_{X'})|_{X'_y}\cong H^0(mK_{X'_y})\otimes O_{X'_y}$ for general $y\in Y$ by Grauert's theorem. By adjunction, we have $O_{X'}(mK_{X'})|_{X'_y}\cong O_{X'_y}(mK_{X'_y})$. Hence the morphism $H^0(mK_{X'_y})\otimes O_{X'_y} \to O_{X'_y}(mK_{X'_y}) $ is surjective, i.e., $K_{X'_y}$ is semi-ample.
\end{proof}

\begin{lem}\label{lem:mmp for parabolic}
Let $f:X\to Y$ be a parabolic fibration. Let $\pi:X\dashrightarrow X'$ be a good minimal model of $X$ over $Y$. Then the parabolic fibrations $f:X\to Y$ and $f':X'\to Y$ have the same moduli and discriminant divisors, and $B_{X'}=\pi_*B_X$. If $B_X$ is effective, then $\pi$ contracts exactly $B_X$.
\end{lem}
\begin{proof}
Since $\pi$ is $K_X$-negative, we have $f_*O_X(mK_{X/Y})=f'_*O_{X'}(mK_{X'/Y})$ for all sufficiently divisible integers $m$. Hence, the sum of moduli and discriminant divisors of $f$ and $f'$ is the same, which implies that $B_{X'}=\pi_*B_X$. For a prime divisor $P\subset Y$, $K_X-\frac{1}{b}B_X+t_Pf^*P$ is numerically trivial over $Y$. Hence the two pairs $(X,-\frac{1}{b}B_X,t_Pf^*P)$, $(X',-\frac{1}{b}B_{X'},t_Pf'^*P)$ are crepant birational to each other. Thus $\lct(X,-\frac{1}{b}B_X,f^*P)=\lct(X',-\frac{1}{b}B_{X'},f'^*P)$, so $f$ and $f'$ have the same discriminant and moduli divisors.

From now on, we assume that $B_X$ is effective. Then $B_{X'}=\pi_*B_X$ is effective. By Lemma \ref{lem:mmp for fibration}, we have $K_{X'_y}\sim_{\Qq}0$ since $\kappa(X'_y)=0$. Then $B_{X'}$ is vertical by $\frac{1}{b}B_{X'}|_{X'_y}=K_{X'_y}\sim_{\Qq}0$. Since the vertical divisor $B_{X'}$ is very exceptional and semi-ample over $Y$, we have $B_{X'}=0$ \cite[Lemma 2.10]{Lai2011Varieties}. By $K_{X}=\pi^*K_{X'}+B_{X}$ and $\pi$ is $K_X$-negative, $\pi$ contracts exactly the divisor $B_X$.
\end{proof}

\begin{proof}[Proof of Theorem \ref{thm:MMP to isotrivial}]
By Lemmas \ref{lem:Eff of BX} and \ref{lem:mmp for parabolic}, for the parabolic fibration $f'$, we have $M_Y\equiv0$, $B_Y=0$, and $B_{X'}=0$. Hence, the theorem follows from \cite[Theorem 4.7 and Proposition 4.4]{Ambro2005TheModuli}.
\end{proof}

\section{Proof of Theorem \ref{thm:sing fiber}}

\begin{lem}\label{lem:exceptional}
Let $f:X\to Y$ be a parabolic fibration between smooth projective varieties. If $M_Y\equiv0$ and $B_Y=0$, then any $f$-exceptional divisor is contained in $B_X$. 
\end{lem}
\begin{proof}
By \cite[Lemma 4.2.2]{Fujino2020BookIitaka}, there exists a commutative diagram

\begin{center}
\begin{tikzcd}
X \arrow[d, "f"] & X' \arrow[d, "f'"] \arrow[l, "p"'] \\
Y                & Y'\arrow[l, "q"']                
\end{tikzcd}    
\end{center}
where $Y'$, $X'$ are smooth projective varieties, $p,q$ are birational morphisms, $f'$ is a fibration, such that any $f'$-exceptional divisor is $p$-exceptional.

Let $M_{Y'}$ be the moduli divisor of the parabolic fibration $f'$.
Write $q^*M_Y+E_3=M_{Y'}$ for some $q$-exceptional divisor $E_3$ which is not necessarily effective. Since $M_{Y'}$ is pseudo-effective and $M_Y\equiv0$, $E_3$ is pseudo-effective. Hence $E_3$ is effective by a lemma due to Lazarsfeld (cf. \cite[Corollary 13]{Kollar2009QuotOfCalabi}). 

Write 
\begin{align*}
K_{X'}=p^*K_X+E_1  \text{ and }  K_{Y'}=q^*K_Y+E_2,
\end{align*}
where $E_1$, $E_2$ are effective exceptional divisors of the birational morphisms $p:X'\to X$ and $q:Y'\to Y$, respectively. By pulling back the canonical bundle formula for $f$:
\begin{align*}
K_X=f^*(K_Y+M_Y)+\frac{1}{b}B_X,    
\end{align*} we have 
\begin{align*}
K_{X'}=f'^*(K_{Y'}+M_{Y'})+E_1-f'^*(E_2+E_3)+\frac{1}{b}p^*B_X.    
\end{align*} 
By comparing this formula with the canonical bundle formula for $f'$:
\begin{align*}
K_{X'}=f'^*(K_{Y'}+M_{Y'}+B_{Y'})+\frac{1}{b}B_{X'},    
\end{align*}
we have $f'^*B_{Y'}+\frac{1}{b}B_{X'}=E_1-f'^*(E_2+E_3)+\frac{1}{b}p^*B_X$. Hence all components of $E_1-f'^*(E_2+E_3)+\frac{1}{b}p^*B_X$ with negative coefficient are exceptional over $Y'$. 

Let $Q\subset X$ be a prime divisor exceptional over $Y$, and let $Q'$ be its strict transform on $X'$. Then $Q'$ is not exceptional over $Y'$. Let $P'=f'(Q')$. Then $P'$ is exceptional over $Y$. Hence $P'\subset E_2$, so $Q'\subset f'^*E_2$. Since $Q'$ is not exceptional over $Y'$, we have \begin{align*}
0\leq\mult_{Q'}(E_1-f'^*(E_2+E_3)+\frac{1}{b}p^*B_X)=\mult_{Q'}(-f'^*(E_2+E_3)+\frac{1}{b}p^*B_X).    
\end{align*} Hence $Q'\subset p^*B_X$ and $Q\subset B_X$. 
\end{proof}

\begin{lem}\label{lem:irreducible}
Let $f:X\to Y$ be a parabolic fibration between smooth projective varieties. If $M_Y\equiv0$ and $B_Y=0$, and if $P\subset Y$ is a prime divisor, then the divisor $f^*P\setminus B_X$ is reduced and irreducible.
\end{lem}
\begin{proof}
Step 1. To prove that $f^*P\setminus B_X$ is reduced and irreducible, we may replace $Y$ by the intersection of general hyperplane sections $C=H_1\cap \cdots \cap H_{\dim Y-1}$ and $X$ by $f^*H_1\cap \cdots \cap f^*H_{\dim Y-1}$.

Step 2. Assume that $f$ is semistable and that $b:=\min\{k\in\Zz_{>0}\mid \dim H^0(kK_F)>0\}=1$, where $F$ is a general fiber. The canonical bundle formula writes $K_X=f^*(K_C+M_C)+B_X$ and $\deg f_*O(K_{X/C})=\deg M_C=0$. Let $Q$ be a component of $f^*P$. By the exact sequence
\begin{align*}
0\to O_X(K_X)\to O_X(K_X+Q)\to O_Q(K_{Q})\to0,    
\end{align*} we have 
\begin{align*}
0\to f_*O_X(K_X)\to f_*O_X(K_X+Q)\to H^0(K_{Q})\otimes\Cc(P)\to R^1f_*O_X(K_X)    
\end{align*}
is exact. Since $R^1f_*O_X(K_X)$ is torsion free by \cite[Theorem 2.1]{Kollar1986Higher1}, and $H^0(K_{Q})\otimes\Cc(P)$ is a skyscraper sheaf, the morphism 
\begin{align*}
H^0(K_{Q})\otimes\Cc(P)\to R^1f_*O_X(K_X)    
\end{align*} is zero. Assume that $f^*P\setminus B_X$ is reducible. Then $B^v_X+Q$ is very exceptional. Hence 
\begin{align*}
f_*O_X(K_X)=f_*O_X(K_X+Q)=O_C(K_C+M_C)    
\end{align*} by Lemma \ref{lem:very exceptional}. Hence $H^0(K_{Q})=0$ for any component $Q$ of $f^*P$. On the other hand, by Theorem \ref{thm:sum of pg}, we have
\begin{align*}
1=\dim H^0(K_F)=\sum_{Q\subset f^*P}\dim H^0(K_Q)=0,\end{align*} which is a contradiction. Hence $f^*P\setminus B_X$ is irreducible.

Step 3. In general, let $\pi:W\to X$ be the canonical cover of $f$ (cf. \cite[Remark 2.6]{FujinoMori2000}). Set $g=f\circ\pi$. Then $g:W\to C$ is a parabolic fibration whose general fiber $\Tilde{F}$ has $\dim H^0(K_{\Tilde{F}})=1$. Let $q:C'\to C$ be a semistable reduction of $g$, and let $g':W'\to C'$ be the induced fibration. We have the following commutative diagram.

\begin{center}
\begin{tikzcd}
W \arrow[d, "\pi"] \arrow[dd, "g"', bend right] & W' \arrow[l, "p"'] \arrow[dd, "g'"] \\
X \arrow[d, "f"]                                &                                     \\
C                                               & C' \arrow[l, "q"']                 
\end{tikzcd}   
\end{center}

Let $M_{C'}$ and $B_{C'}$ be the moduli divisor and the discriminant divisor of $g'$, respectively. By \cite[Corollary 2.5 and Lemma 3.4]{FujinoMori2000}, $M_{C'}=q^*M_C\equiv0$. Since $g'$ is semistable, we have $B_{C'}=0$. Write the canonical bundle formula of $f$ as 
\begin{align*}
K_X=f^*(K_C+M_C)+\frac{1}{b}B_X.    
\end{align*}
By pulling back along $\pi\circ p$, we have 
\begin{align*}
K_{W'}=g'^*(q^*K_C+M_{C'})+\frac{1}{b}(\pi\circ p)^*B_X+R,
\end{align*} where $R=K_{W'}-(\pi\circ p)^*K_X$ is the ramification divisor of the generically finite morphism $\pi\circ p$.

Fix a point $P'\in C'$ with $q(P')=P$ and $q^*P=mP'$ near $P'$. Then near $P'$, we have 
\begin{align*}
K_{W'}=g'^*(K_{C'}+M_{C'})+\frac{1}{b}(\pi\circ p)^*B_X+R-(m-1)g'^*P'.    
\end{align*}
By comparing it with the canonical bundle formula of $g'$:
\begin{align*}
K_{W'}=g'^*(K_{C'}+M_{C'})+B_{W'},    
\end{align*} we have 
\begin{align*}
B_{W'}=\frac{1}{b}(\pi\circ p)^*B_X+R-(m-1)g'^*P'.    
\end{align*}

Let $Q$ be a component of $f^*P\setminus B_X$. Then its coefficient in $f^*P\setminus B_X$ is $1$ by $\lct(X,-\frac{1}{b}B_X,f^*P)=1$. Let $Q'$ be a component of $g'^*P'$ such that $\pi\circ p(Q')=Q$, then $(\pi\circ p)^*Q=mQ'$ near $Q'$. Since near $Q'$, $R=(m-1)Q'$, we have $Q'\not\subset B_{W'}$. By Step 2, we know that $g'^*P'\setminus B_{W'}$ is irreducible. Hence $f^*P\setminus B_{X}$ is irreducible. It is reduced by $\lct(X,-\frac{1}{b}B_X,f^*P)=1$.
\end{proof}

\begin{proof}[Proof of Theorem \ref{thm:sing fiber}]
This follows from Lemmas \ref{lem:Eff of BX}, \ref{lem:exceptional}, and \ref{lem:irreducible}.
\end{proof}

\section{Proof of Theorem \ref{thm:sum of pg}}
We recall mixed Hodge structures that appear in the degeneration of varieties. Our main reference is \cite{TopicsTranscendental}, especially Chapter VI.

A weight $m$ Hodge structure $(F,H_{\Zz})$ consists of a finite rank lattice $H_{\Zz}$ and a decreasing filtration (the Hodge filtration) on $H:=H_{\Zz}\otimes\Cc$: $0\subset F^m\subset F^{m-1}\subset\cdots\subset F^0=H$ such that $H=F^p\oplus \overline{F^{m-p+1}}$. For a Hodge structure $H_{\Zz}$, we set $H^{p,q}:={F^p}/{F^{p+1}}$ ($q=m-p$). A polarization of a Hodge structure $H_{\Zz}$ is a bilinear form $Q$ on $H$, which is symmetric for $m$ even, skew-symmetric for $m$ odd, and satisfying $Q(H^{p,q},H^{p',q'})=0$ unless $p=q'$, $q=p'$ and $i^{p-q}Q(s,\overline{s})>0$ if $0\neq s\in H^{p,q}$.

Let $Y$ be a complex manifold. A weight $m$ variation of Hodge structure $(F,H_{\Zz})$  on $Y$ consists of a local system of finite rank lattice $H_{\Zz}$ and a decreasing filtration (the Hodge filtration) of holomorphic subbundles of $H:=H_{\Zz}\otimes O_{Y}$: $0\subset F^m\subset F^{m-1}\subset \cdots\subset F^0=H$, such that at each point of $Y$, the fiber of $(F,H_{\Zz})$ is a weight $m$ Hodge structure, and satisfying Griffiths transversality: $\nabla F^p\subset F^{p-1}\otimes\Omega^1_Y$ for all $p$, where $\nabla$ is the flat connection on $H$. A polarization $Q$ of a variation of Hodge structure $(F,H_{\Zz})$ is a flat bilinear form $Q$ on $H$, such that at each point of $Y$, the fiber of $(F,H_{\Zz},Q)$ is a polarized Hodge structure.

Let $(F,H_{\Zz},Q)$ be a polarized variation of Hodge structure on a complex manifold. We call $H^{p,q}=F^p/F^{p+1}$ the Hodge bundles and call the metric on $H^{p,q}$ induced by the polarization $Q$ the Hodge metric. 

\begin{exa}\label{example1}
Let $X$ be a smooth projective variety of dimension $n$. 

Let $F^p=\oplus_{k\geq p}H^{k,m-k}(X)$, $0\leq p\leq m$, where $H^{k,m-k}(X)=H^{m-k}(X,\Omega_X^k)$. Then $(F,H^m(X,\Zz)/\text{torsion})$ is a weight $m$ Hodge structure.

Let $\omega$ be the positive $(1,1)$-form induced by an ample line bundle. Let $P^m(X,\Cc):=H^m(X,\Cc)_{\text{prim}}$ be the primitive cohomology. Let $P^m(X,{\Zz}):=H^m(X,\Zz)\cap P^m(X,\Cc)$. Let $P^{p,q}(X):=H^{p,q}(X)\cap P^m(X,\Cc)$, $p+q=m$, $0\leq p\leq m$. Let $F^p=\oplus_{k\geq p}P^{k,m-k}(X)$, $0\leq p\leq m$. Let $Q$ be the following bilinear form on $P^m(X,\Cc)$: $Q(s,t)=(-1)^{{m(m-1)}/{2}}\int_Xs\wedge t\wedge \omega^{n-m}$. Then $(F, P^m(X,\Zz),Q)$ is a weight $m$ polarized Hodge structure. 
\end{exa}

\begin{exa}\label{example2}
Let $f:X\to Y$ be a smooth projective morphism between complex manifolds, with fiber dimension $n$. Let $\omega$ be the fiberwise positive $(1,1)$-form induced by a relative ample line bundle. Let $H_{\Zz}$ be a local system of lattices, let $F$ be a filtration of holomorphic subbundles of $H:=H_{\Zz}\otimes O_Y$, and let $Q$ be a flat bilinear form on $H$, such that at each point $y\in Y$,
the fiber of $(F,H_{\Zz},Q)$ is the polarized Hodge structure $(F, P^m(X_y,\Zz),Q)$ for $(X_y,\omega|_{X_y})$ given in the previous example. Then $(F,H_{\Zz},Q)$ is a weight $m$ polarized variation of Hodge structure. We also write this polarized variation of Hodge structure simply as $(R^mf_*\Zz_X)_{\text{prim}}$. Note that if $m=n$, then $F^n=f_*O_X(K_{X/Y})$.
\end{exa}

A mixed Hodge structure $(F,W,H_{\Zz})$ consists of a finite rank lattice $H_{\Zz}$, an increasing filtration (the weight filtration) $W$ on $H_{\Qq}:=H_{\Zz}\otimes\Qq$, and a decreasing filtration (the Hodge filtration) $F$ on $H:=H_{\Zz}\otimes\Cc$, such that each $\gr^W_m:=W_m/W_{m-1}$ together with the filtration induced by $F$ is a Hodge structure of weight $m$. More precisely, $F^p\gr_m^W={F^p\cap W_m\otimes\Cc}/{F^p\cap W_{m-1}\otimes\Cc}$. For a mixed Hodge structure $(F,W,H_{\Zz})$, we set $I^{p,q}H:=(\gr_{p+q}^W)^{p,q}=F^p\gr_{p+q}^W/F^{p+1}\gr_{p+q}^W$.

Let $f:X\to \Delta$ be a semistable fibration over a disc $\Delta\subset\Cc$ such that $f$ is smooth over $\Delta\setminus\{0\}$. Let $n$ be the fiber dimension. Fix a smooth fiber $X_{t_0}$, $t_0\in \Delta\setminus\{0\}$ and let $X_0=\sum_iX_i$ be the central fiber. Then
\begin{enumerate}
    \item there is a mixed Hodge structure on $H^n(X_0):=H^n(X_0,\Cc)$, which satisfies 
    \begin{align}\label{pg of snc}
     I^{n,0}H^n(X_0)=\oplus_i H^{n,0}(X_i), \text{ where } H^{n,0}(X_i)=H^0(K_{X_i}),  
    \end{align} 
    \item there is a mixed Hodge structure $(F_{\infty},W, H^n_{\infty})$ on $H^n(X_{t_0}):=H^n(X_{t_0},\Cc)$, called the limit mixed Hodge structure, and
    \item (as a corollary of the Clemens--Schmid sequence) there is an isomorphism of vector spaces  \cite[p.118, third row from the bottom]{TopicsTranscendental}:
    \begin{align}\label{iso of mixed}
     I^{n,0}H^n(X_0)\cong I^{n,0}H^n_{\infty}.   
    \end{align} 
\end{enumerate}
Since the proof of Theorem \ref{thm:sum of pg} depends on the construction of $(F_{\infty},W, H^n_{\infty})$, we briefly sketch it here. Since the restriction $X\setminus f^{-1}\{0\}\to\Delta\setminus\{0\}$ is a smooth fibration, the fundamental group $\pi_1(\Delta\setminus\{0\},t_0)$ acts on $H^n(X_{t_0},\Qq)$. Let $T$ be the automorphism (the monodromy transformation) of $H^n(X_{t_0},\Qq)$ induced by the generator of $\pi_1(\Delta\setminus\{0\},t_0)$. The monodromy theorem  \cite{Landman1973Picard} states that $(T-I)^{n+1}=0$. Set \[N:=\log T=(T-I)-\frac{1}{2}(T-1)^2+\cdots+(-1)^{n+1}\frac{1}{n}(T-I)^n,\] then $N^{n+1}=0$. 

\textbf{Construction of $W$.} Define 
\begin{align*}
W_0=\Ima N^n, W_{2n-1}=\ker N^n, \text{ and } W_{2n}=H^n(X_{t_0},\Qq). \end{align*} Then on $W_{2n-1}/W_0$, we have $(N|_{W_{2n-1}/W_0})^n=0$. We define $W_1$, $W_{2n-2}$ as subspaces of $H^n(X_{t_0},\Qq)$ such that 
\begin{align*}
W_1/W_0=\Ima (N|_{W_{2n-1}/W_0})^{n-1} \text{ and } W_{2n-2}/W_0=\ker (N|_{W_{2n-1}/W_0})^{n-1}.    
\end{align*}Continuing this process, we obtain the filtration 
\begin{align*}
 0\subset W_0\subset W_1\subset\cdots\subset  W_{2n-2}\subset W_{2n-1}\subset W_{2n}=H^n(X_{t_0},\Qq).
 \end{align*} 
We claim that 
 \begin{align}\label{kerN<W_n}
 \ker N\subset W_n.    
 \end{align} 
Let $s\in \ker N$. If $s\in W_{2n-k}$ for some $0\leq k\leq n-1$, then by $W_{2n-k-1}/W_{k-1}=\ker(N|_{W_{2n-k}/W_{k-1}})^{n-k}$, we have $s\in W_{2n-k-1}$. Hence, the claim follows by induction on $k$.

\textbf{Construction of $F_{\infty}$.} 
Let $\pi:U=\{z\in\Cc\mid \text{Im}(z)>0\}\to \Delta\setminus\{0\},z\mapsto e^{2\pi i z}$ be the universal cover. Fix a point $z_0\in \pi^{-1}(t_0)$. For each $z\in U$, let $\gamma_z$ be a path from $z$ to $z_0$, let $F'(z)$ be the usual Hodge filtration of $H^n(X_{\pi(z)},\Cc)$ (see Example \ref{example1}), and let $F(z)$ be the parallel transport of $F'(z)$ along the path $\pi\circ\gamma_z$. Then the resulting $F(z)$ is a filtration on $H^n(X_{t_0},\Cc)$ and satisfies $F(z+1)=TF(z)$. Note that $e^{-zN}F(z)=e^{-(z+1)N}F(z+1)$. In \cite{Schmid1973singularities}, Schmid proved that the limit $\lim_{\text{Im}( z)\to\infty} e^{-zN}F(z)$ exists and $F_{\infty}$ is defined to be the limit.

\begin{proof}[Proof of Theorem \ref{thm:sum of pg}]

Let $C_0\subset C$ be the locus where $f$ is smooth, $X_0:=f^{-1}C_0$, and $f_0:=f|_{X_0}$. Consider the polarized variation of Hodge structure $H_{\Zz}=(R^nf_{0*}\Zz_{X_0})_{\text{prim}}$ (Example \ref{example2}), and let $H=H_{\Zz}\otimes O_{C_0}$, where $n$ is the fiber dimension. Let $\nabla$ be the flat connection on $H$, and let $0\subset F^n\subset F^{n-1}\subset \cdots\subset F^0=H$ be the Hodge filtration. We have Griffiths transversality: $\nabla F^p\subset F^{p-1}\otimes\Omega_{C_0}^1$ for each $p$. Recall that $F^n={f_0}_*O(K_{X_0/C_0})$. $\nabla$ induces a $O_{C_0}$-linear map $\psi:F^n\to (F^{n-1}/F^{n})\otimes\Omega^1_{C_0}$. 

Let $\Theta$ be the curvature form induced by the Hodge metric $h$ on $F^n={f_0}_*O(K_{X_0/C_0})$. Since $f$ is semistable, the monodromy transformations of $H$ around points of $C\setminus C_0$ are unipotent by the monodromy theorem. Hence the integral formula (cf. \cite[Lemma 21]{Kawamata1981Characterization}) gives
\begin{align*}
\deg f_*O_X(K_{X/C})=\frac{i}{2\pi}\int_C\text{Tr }\Theta.    
\end{align*}
Since $\deg f_*O_X(K_{X/C})=0$ and $\Theta$ is Griffiths semi-positive, we have $\Theta=0$. By the formula computing $\Theta$: $h(\Theta s,t)=h(\psi s,\psi t)$, where $s$, $t$ are smooth sections of $F^n$ (cf. \cite[Chapter II Proposition 4]{TopicsTranscendental}), we deduce that $\nabla F^n\subset F^n\otimes \Omega^1_{C_0}$, i.e., $F^n={f_0}_*O(K_{X_0/C_0})$ is a flat subbundle of $H$.

Fix a small disc $\Delta\subset C$ centered at a point of $C\setminus C_0$. We may assume that $t_0\in \Delta\setminus \{0\}$. Let $T\in\Aut H^n(X_{t_0},\Cc)$ be the monodromy transformation around the center of $\Delta$. Since ${f_0}_*O(K_{X_0/C_0})$ is a flat subbundle of $H$, $T$ acts on the fiber of ${f_0}_*O(K_{X_0/C_0})$ at $t_0$, i.e., $T$ acts on the subspace $H^0(K_{X_{t_0}})\subset H^n(X_{t_0},\Cc)$. 
Since $T$ is unipotent and preserves the Hodge metric on $H^0(K_{X_{t_0}})$, we have $T|_{H^0(K_{X_{t_0}})}$ is trivial.

Following the notation in the definition of $F^n_{\infty}$, for a point $z\in U$, $F^n(z)$ is the parallel transport of $H^0(K_{X_{\pi(z)}})$ from $\pi(z)$ to $t_0$ along a path. Since ${f_0}_*O(K_{X_0/C_0})$ is a flat subbundle of $H$, the parallel transport of $H^0(K_{X_{\pi(z)}})$ from $\pi(z)$ to $t_0$ is just the space $H^0(K_{X_{t_0}})$, i.e., we have $F^n(z)=H^0(K_{X_{t_0}})$ for all $z\in U$. Since $T$ acts trivially on $H^0(K_{X_{t_0}})$, we have
\begin{align*}
F^n_{\infty}=\lim_{\text{Im}(z)\to\infty}e^{-zN}F^n(z) =\lim_{\text{Im}(z)\to\infty}e^{-zN}H^0(K_{X_{t_0}})=H^0(K_{X_{t_0}}).  
\end{align*} 
Hence $F^n_{\infty}\subset\ker N$. By \ref{kerN<W_n}, we have $F^n_{\infty}\subset W_n\otimes\Cc$. Hence $I^{n,i}H^{n}_{\infty}={F^n_{\infty}\cap W_{n+i}\otimes\Cc}/{F^n_{\infty}\cap W_{n+i-1}\otimes\Cc}=0$ for all $i\geq1$. Hence $\dim F^{n}_{\infty}=\sum_i\dim I^{n,i}H^n_{\infty}=\dim I^{n,0}H^{n}_{\infty}$ and $\dim I^{n,0}H^n(X_0)=\dim I^{n,0}H^n_{\infty}=\dim F^n_{\infty}=\dim H^0(K_{X_{t_0}})$, where the first equality follows from \ref{iso of mixed}. By \ref{pg of snc}, we conclude that 
\begin{align*}
\sum_i\dim H^0(K_{X_i})= \dim H^0(K_{X_{t_0}}).
\end{align*}
\end{proof}


\subsection*{Acknowledgments}
The author thanks Professor Zhan Li for helpful discussions and encouragement. He thanks the anonymous reviewer whose suggestions improved the article. This work was supported by Professor Mao Sheng through the CAS Project for Young Scientists in Basic Research (Grant No. YSBR-032).

\bibliographystyle{alpha}

\bibliography{bibfile}

\end{document}